\documentclass[12pt,letterpaper]{amsart}
\usepackage{url}
\usepackage{fullpage}
\usepackage{pslatex} 
\usepackage{amsmath}
\usepackage{amssymb}
\usepackage{amsthm}
\usepackage{color}
\usepackage{array}
\usepackage{xy}
\usepackage{setspace}
\usepackage{multicol}
\usepackage[pdftex]{graphicx}

\newcommand{\id}{\operatorname{id}}
\newcommand\numberthis{\addtocounter{equation}{1}\tag{\theequation}}

\newtheoremstyle{slanted}
{3pt}
{3pt}
{\slshape}
{}
{\bfseries}
{.}
{.5em}
{}
\theoremstyle{slanted}
\newtheorem{thm}{Theorem}[section]
\newtheorem{lem}[thm]{Lemma}

\newtheorem{cor}[thm]{Corollary}
\newtheorem{conj}[thm]{Conjecture}
\newtheorem{conv}[thm]{Convention}
\newtheorem{defn}[thm]{Definition}
\theoremstyle{remark}
\newtheorem{rem}[thm]{Remark}
\newtheorem{ex}[thm]{Example}

\begin{document}

\date{}
\title[Equivalence Classes in $S_n$ for Three Families of Pattern-Replacement Relations]{Equivalence Classes in $S_n$ for Three Families of Pattern-Replacement Relations}
\author{William Kuszmaul}
\author{Ziling Zhou}
\maketitle
\begin{abstract}
We study a family of equivalence relations on $S_n$, the group of permutations on $n$ letters, created in a manner similar to that of the Knuth relation and the forgotten relation. For our purposes, two permutations are in the same equivalence class if one can be reached from the other through a series of pattern-replacements using patterns whose order permutations are in the same part of a predetermined partition of $S_c$. In particular, we are interested in the number of classes created in $S_n$ by each relation and in characterizing these classes.

Imposing the condition that the partition of $S_c$ has one nontrivial part containing the cyclic shifts of a single permutation, we find enumerations for the number of nontrivial classes. When the permutation is the identity, we are able to compare the sizes of these classes and connect parts of the problem to Young tableaux and Catalan lattice paths. 

Imposing the condition that the partition has one nontrivial part containing all of the permutations in $S_c$ beginning with $1$, we both enumerate and characterize the classes in $S_n$. We do the same for the partition that has two nontrivial parts, one containing all of the permutations in $S_c$ beginning with $1$, and one containing all of the permutations in $S_c$ ending with $1$. 
\end{abstract}

\section{Introduction}
In 1970, the Robinson-Schensted-Knuth (RSK) correspondence brought the so-called Knuth relation to the forefront of mathematics.  This is an equivalence relation on permutations (or, more generally, words) which connects two permutations if one can be transformed into the other by a given set of rules that allow the switching of neighboring letters (as in bubblesort) under particular conditions \cite{Kn}. An analogue of this equivalence, the so-called forgotten relation \cite{NS}, later emerged, sharing some algebraic applications with the Knuth relation. Both the Knuth and the forgotten relation have a common structure: They are transitive, reflexive, symmetric closures of relations given by allowing the rearrangement of blocks of adjacent letters in a permutation from a certain order into another given order. This inspired various authors \cite{LPRW}, \cite{PRW}, \cite{Kuszmaul12} to systematically analyze relations of the same type, leading to numerous results on the number of equivalence classes and their sizes. This paper extends the study to consider pattern-replacement relations using patterns of arbitrary length. To our knowledge, we are the first to analyze an infinite family of pattern-replacement relations.

Before we present our results, we describe permutation
pattern-replacement equivalences in more detail. Given a set $S
\subseteq S_c$ of patterns (which we call our \emph{replacement set}),
such as $S = \{123, 321\}$, we can define an equivalence relation on
$S_n$ as follows. Whenever a permutation $w$ contains some pattern $p
\in S$ (meaning there is a contiguous subword of $w$ which is order
isomorphic to $p$), we allow ourselves to rearrange the letters within
that pattern to form any other pattern in $S$. The resulting
permutation $w'$ is said to be \emph{equivalent} to $w$ under
\emph{$S$-equivalence}. For example, if $S = \{123, 321\}$, then the
letters $134$ form a $123$ pattern in the permutation $51342$, and
thus we can rearrange them to form a $321$ pattern in $54312$. This
means that $51342$ is equivalent to $54312$ under $\{123,
321\}$-equivalence. We can continue like this, noting that $543$ forms
a $321$ pattern in $54312$; rearranging it to form a $123$ pattern, we
get $34512$. This means that $51342$, $54312$, and $34512$ are all
equivalent under $\{123, 321\}$-equivalence. Moreover, one can check
that no other permutations are equivalent to them, meaning that these
three permutations form an equivalence class. In general, given a set
$S$, we wish to study the equivalence classes in $S_n$ under
$S$-equivalence.

Rather than just considering a single set $S$, one can consider a
collection $P$ of disjoint subsets of $S_c$. (We say $P$ is a
\emph{replacement partition}.) Two permutations $x$ and $y$ are
equivalent under $P$-equivalence if $y$ can be reached from $x$ by a
series of pattern replacements in which patterns are only replaced
with other patterns which are in the same part in $P$. For example,
under $\{123, 321, 231\}\{213, 132\}$-equivalence, one is allowed to
replace $123$ patterns, $321$ patterns, and $231$ patterns with each
other; and to replace $213$ patterns and $132$ patterns with each
other.

Most of these conventions follow closely those of \cite{LPRW},
\cite{PRW}, and \cite{Kuszmaul12}. Additionally, we will often use the
word \emph{hit} in place of \emph{pattern} when referring to $c$
consecutive letters in a word which form a pattern; this allows us to
disambiguate the letters forming the pattern from the actual pattern
in $S_c$. For a replacement set or replacement partition $K$, we will
use \emph{$K$-hit} to refer to an instance of some pattern in $K$. 

 This paper considers infinite families of $K$-equivalences. The
 families we consider are the following:
\begin{itemize}
\item $K$ is the replacement set containing a single permutation and
  all of its cyclic shifts. In this case, we study the nontrivial
  classes in Section \ref{subsecrot}. (A class is called
  \emph{trivial} if it contains a single permutation.)
\item $K$ is the replacement set containing all permutations in $S_k$
  starting with $1$. Here, we enumerate and characterize the classes
  in Section \ref{subsecstartwith1}. (This solves a previously open
  case for $k=3$.)
\item $K$ has exactly two parts -- one of them containing all
  permutations in $S_k$ starting with $1$, and the other containing
  all permutations in $S_k$ ending with $1$. Again, we are able to
  enumerate and characterize the classes (Section
  \ref{subsecstartandendwith1}). (This again solves a previously open
  case for $k=3$.)
\end{itemize} 

The centerpiece is Section \ref{subsecrot} with the treatment of the
cyclic-shifts case. The main enumeration in this case leads us to
consider Young tableaux and Catalan paths. Section \ref{secconclusion}
concludes with a discussion of open problems and future work.

\section{Replacement Partitions Using Cyclic Shifts}\label{subsecrot}
 In this section, we consider replacement sets containing a
 permutation in $S_c$ and its cyclic shifts. In the case of $c=3$, this was
 previously studied by \cite{PRW}. Thus when $c=3$, our results serve
 as alternative proofs for several of theirs.

\begin{defn}
The \emph{cyclic shifts} of a permutation $m \in S_c$ are the
permutations of the form $ab$ for some $a$ and $b$ such that $m =
ba$. Note that either $a$ or $b$ may be the empty word.

Moreover, the cyclic shift beginning with the $j$-th letter of $m$ is
denoted $m^{+j}$.
\end{defn}

Observe that the cyclic shifts of $m$ can be obtained by repeatedly
killing the final letter of $m$ and reappending to $m$'s
beginning. For example, the cyclic shifts of $6172534$ are

\begin{center}
\begin{tabular}{r c c c c c c c c c c c c c c}
$m^{+1} =$&6&1&7&2&5&3&4&&&&&&& \\
$m^{+2} =$&&1&7&2&5&3&4&6&&&&&& \\ 
$m^{+3} =$&&&7&2&5&3&4&6&1&&&&& \\
$m^{+4} =$&&&&2&5&3&4&6&1&7&&&& \\
$m^{+5} =$&&&&&5&3&4&6&1&7&2&&& \\
$m^{+6} =$&&&&&&3&4&6&1&7&2&5&& \\
$m^{+7} =$&&&&&&&4&6&1&7&2&5&3&
\end{tabular}
\end{center}

Throughout the section we use $m$ to denote a permutation of size $c$
(i.e., in $S_c$). We then use $K$ to denote the pattern replacement
set containing the cyclic shifts of $m$.

\begin{ex} If $m=13542$, then $K=\{13542, 35421, 54213, 42135, 21354\}$.
\end{ex}

Our goal is to characterize the nontrivial classes in $S_n$, that is
the classes of size greater than one, under $K$-equivalence. Surprisingly, we
find that regardless of one's choice for $m \in S_c$ and $n \in
\mathbb{N}$, there are never more than two nontrivial equivalence
classes in $S_n$ under $K$-equivalence (Theorem \ref{thm1}). Moreover,
when $c$ is odd, parity is maintained by $K$-rearrangements, implying
that for $n>c$ there are two nontrivial classes, one for each parity
of non $K$-avoiders in $S_n$. 

We then restrict ourselves to the case of $m$ being the identity in
$S_c$ (and $n>c$). When $c$ is even, we find that there is only one
equivalence class in $S_n$ (Corollary \ref{singleclasscor}). When $c$
is odd and $n$ is even, we show that there are two equi-sized
nontrivial equivalence classes in $S_n$ (Theorem \ref{even n
  rot}). When $c$ is odd, $n$ is odd, and $c > n/2$ we express the
difference between the sizes of the two nontrivial classes in $S_n$ in
terms of the Catalan number (Theorem \ref{difference in parity of odd
  n}). For the remaining case of $c$ odd, $n$ odd, and $c < n / 2$, we
reduce the problem of comparing the nontrivial class sizes in $S_n$ to
a problem having to do with hit-huggers (Lemma \ref{hit-ended -->
  non-avoiding}).

We begin by showing that for $n > c$, and for any choice of $m \in
S_c$, there are at most two non-trivial equivalence classes in $S_n$
under $K$-equivalence. To do this, we find two permutations $p_n$
and $q_n$ so that every non $K$-avoiding permutation in $S_n$ is
equivalent to one of them.

In particular, let $p_n$ denote the permutation constructed as
follows. Take the identity permutation in $S_n$ and rearrange the
letters $2,\ldots,c+1$ to form $m$. Then construct $q_n$ by swapping
the letters $1$ and $2$ in $p_n$.

\begin{ex} If $m = 23541$, then $p_{9} = 134652789$ and $q_9 = 234651789$.
\end{ex}

\begin{defn}
A permutation in $S_n$ is \emph{$pq$-reachable} if it is equivalent
  to one of $p_n$ or $q_n$ under $K$-equivalence.
\end{defn}

We want to show that the $pq$-reachable permutations are exactly the
non $K$-avoiding permutations in $S_{n>c}$. 

\begin{defn}
Let $w \in S_n$ for some $n$, and $i \in \mathbb{N}$ be so that $1 \le
i \le n + 1$. We denote by $i \rightharpoonup w$ the permutation in
$S_{n+1}$ obtained as follows. Increasing by $1$ each letter of $w$
which is greater than or equal to $i$ in value. Then append $i$ to the
left end of $w$.

Similarly, $w \leftharpoonup i$ is the word obtained by increasing by
$1$ every letter of $w$ which is greater than or equal to $i$, and
then appending $i$ to the right of $w$.
\end{defn}

\begin{ex} Observe that $4 \rightharpoonup 123456 = 4123567$.
\end{ex}

Note that when $n=c+1$, $p_n = 1 \rightharpoonup m$ and $q_n = 2
\rightharpoonup m$. Also note for all $n$ that if $w \equiv v$ for two
permutations $w,v \in S_n$, then $i \rightharpoonup w \equiv i
\rightharpoonup v$ and $w \leftharpoonup i \equiv v \leftharpoonup i$.

We begin with a Lemma allowing us to move around hits within a
permutation.

\begin{lem}\label{harplem}
For any $j \in \{1,2,\ldots,c\}$, we have $j \rightharpoonup m
\equiv m \leftharpoonup (j+1)$ and $m \leftharpoonup j \equiv (j+1)
\rightharpoonup m$ under $K$-equivalence.
\end{lem}
\begin{proof}
Let $a \in \{1,2,\ldots,c\}$ be such that $m_{a-1}=j$ (with indices
cyclic modulo $c$). Then, it is easy to see that $j \rightharpoonup
m^{+a} = m^{+(a-1)} \leftharpoonup (j+1)$. As a consequence, $j
\rightharpoonup m \equiv j \rightharpoonup m^{+a} = m^{+b}
\leftharpoonup (j+1) \equiv m \leftharpoonup (j+1)$. This proves the
first part of the lemma and the second follows by symmetry.
\end{proof}

Next we consider the simple case of $K$-equivalence acting on $S_{c+1}$.
\begin{lem}\label{lem1}
All non $K$-avoiders in $S_{c+1}$ are $pq$-reachable.
\end{lem}
\begin{proof}
Let $w \in S_{c+1}$ be a non $K$-avoider. Then $w$ is of the form
either $i \rightharpoonup m^{+a}$ or $m^{+a} \leftharpoonup i$ for
some $i\in \{1,2,\ldots, c + 1\}$ and some $a\in \{1,2,\ldots,c
\}$. Thus either $w \equiv i \rightharpoonup m$ or $w \equiv m
\leftharpoonup i$ respectively. Hence we only need to prove that $i
\rightharpoonup m$ and $m \leftharpoonup i$ are
$pq$-reachable. Repeatedly using Lemma \ref{harplem}, we obtain the
chains of equivalences
\begin{eqnarray}
1 \rightharpoonup m \equiv m \leftharpoonup 2 \equiv 3 \rightharpoonup m \equiv m \leftharpoonup 4 \equiv \cdots, \label{lem1.pf.3} \\
m \leftharpoonup 1 \equiv 2 \rightharpoonup m \equiv m \leftharpoonup 3 \equiv 4 \rightharpoonup m \equiv \cdots. \label{lem1.pf.4}
\end{eqnarray}

For a given $i$, both $i \rightharpoonup m$ and $m \leftharpoonup i$
appear in one of the chains, and thus each is equivalent to either $1
\rightharpoonup m = p_{c+1}$ or $2 \rightharpoonup m = q_{c+1}$.
\end{proof}

\begin{cor}\label{singleclasscor}
If $c$ is even and $m = \id_c$, then there is only one nontrivial class in $S_{c+1}$ under $K$-equivalence.
\end{cor}
\begin{proof}
We follow the proof of Lemma \ref{lem1}, but notice that now the
equivalence chains (\ref{lem1.pf.3}) and (\ref{lem1.pf.4}) are merged
because the identity in $S_n$ can be written as either $1
\rightharpoonup m$ (an element of the chain (\ref{lem1.pf.3})) or $m
\leftharpoonup n$ (an element of the chain (\ref{lem1.pf.4})).
\end{proof}

Using Lemma \ref{lem1} as a base case, we can now prove that all non
$K$-avoiders are $pq$-reachable under $K$-equivalence.
\begin{thm}\label{thm1}
When $n > c$, all non $K$-avoiders in $S_n$ are $pq$-reachable under $K$-equivalence.
\end{thm} 
\begin{proof}
We will prove this by inducting on $n$, with Lemma \ref{lem1} serving as a base case (the case $n = c + 1$). Assume that the theorem holds for $S_{n-1}$ and that $n \ge c+2$. Let $x\in S_n$ be a non $K$-avoider. If $x$ only has a $K$-hit in its right-most $c$ letters, then applying Lemma \ref{harplem} to the right-most $c+1$ letters, we slide the $K$-hit to the left, placing it in the left-most $n-1$ letters. Thus without loss of generality, $x$ has a $K$-hit in its left-most $n-1$ letters. Applying the inductive hypothesis to the left-most $n-1$ letters of $x$, one can get a permutation $x'$ of the form $a \cdots u$ where $a \in \{ 1,2,3 \}$, $u$ is the right-most letter of $x$, and the letters not shown are in increasing order, except with the first $c$ of them rearranged to form a $K$-hit. Noting that the contiguous subsequence $x'_2 x'_3 \cdots x'_{c+1}$ forms a $K$-hit, and applying the inductive hypothesis to the right-most $n-1$ letters of $x'$, one can get a permutation $x''$ of the form $a b \cdots$ where $a,b\in \{ 1,2,3 \}$ and the letters not shown start off with a $K$-hit containing the letters up to $c+2$ and proceed with the remaining $n-c-2$ letters in increasing order. In showing that all such permutations are $pq$-reachable, it can be assumed without loss of generality that $n = c+2$ (as otherwise, we can simply restrict ourselves to considering the left-most $c+2$ letters). Given this assumption, we may rewrite $x''$ as $j \rightharpoonup (k \rightharpoonup m)$ for some $j \in\{1,2,3\}$ and $k\in\{1,2\}$.

In other words, we only need to show that the 6 permutations in
$S_{c+2}$ of the form $j \rightharpoonup (k \rightharpoonup m)$ for $j
\in\{1,2,3\}$ and $k\in\{1,2\}$ are $pq$-reachable. Let $w = j
\rightharpoonup (k \rightharpoonup m)$ be such a permutation. If $k
\not\equiv n-1 \mod 2$, then applying chains (\ref{lem1.pf.3}) and
(\ref{lem1.pf.4}) we go from $w$ to $j \rightharpoonup (m
\leftharpoonup (n-1)) = (j \rightharpoonup m) \leftharpoonup n$, and
then to $(j' \rightharpoonup m) \leftharpoonup n$ for some $j' \in
\{1,2\}$, thus reaching $p_{c+2}$ or $q_{c+2}$. If $k\equiv n-1\mod 2$
and $c+2 \ge 8$, we apply the chains to go from $w$ to $j
\rightharpoonup (m \leftharpoonup k') = (j \rightharpoonup m)
\leftharpoonup (k'+1)$ for some $k' \in \{3,4\}$ with $k' \not\equiv n
- 1 \mod 2$, then to $(j' \rightharpoonup m) \leftharpoonup (k'+1) =
(j'+1) \rightharpoonup (m \leftharpoonup (k'+1))$ for some $j' \in
\{5, 6\}$, then to $(j'+1) \rightharpoonup (m \leftharpoonup (n-1)) =
((j'+1) \rightharpoonup m ) \leftharpoonup n$, and to $(j''
\rightharpoonup m) \leftharpoonup n$ for some $j'' \in \{1,2\}$, thus
reaching $p_{c+2}$ or $q_{c+2}$. The case of $c < 6$ with $k \equiv
n-1 \mod 2$ can easily be checked by computer. For every possibility
for $m$ (up to rotational and reflectional symmetry) in $S_{c < 6}$,
one simply constructs the equivalence classes in $S_{c+2}$ and
verifies that the classes contain the desired
permutations. Alternatively, for $c \neq 3$, one can also use an
argument similar to but slightly more complicated than the one above.
\end{proof}

\begin{rem}
As a result of Theorem \ref{thm1}, we can start to characterize the classes created under $K$-equivalence in $S_n$ for $n>c$. Indeed for odd $c$, because parity is an invariant, there are \emph{always} two nontrivial classes in $S_n$, one for even permutations, and one for odd. For even $c$, there are always $\leq 2$ nontrivial classes. Observe that if $m = \id_c$, then by Corollary \ref{singleclasscor}, $p_n \equiv q_n$, so there is only one nontrivial class. But this is not the case for all $m$. For example, when $m = 145236$, there appears to always be two nontrivial classes in $S_n$.
\end{rem}

This brings us to the following conjecture, which we already know to be true when $c$ is odd.
\begin{conj}
For a given $m \in S_c$, either:
\begin{itemize}
\item $K$-equivalence yields two nontrivial classes in $S_n$ for all $n > c$, or
\item $K$-equivalence yields one nontrivial class in $S_n$ for all $n > c$.
\end{itemize}
\end{conj}

At this point, we are ready to focus on a particularly interesting case, when $m = \id_c$ and $c$ is odd. Since there are always two nontrivial classes in this case, one for even permutations and one for odd permutations, we would like to know their relative sizes. 

Thus from here until the end of this section, we let $m = \id_c$ and $c$ be odd.

First, two auxiliary notions:

\begin{defn}
A \emph{hit-opener} is a permutation with a $K$-hit starting at the first letter (i.e., left-most letter) and with no other $K$-hits. 
\end{defn}
\begin{defn} A \emph{hit-hugger} is a permutation with only two $K$-hits, one starting at the first letter and one starting at the $(n-c+1)$-th letter.
\end{defn}

\begin{defn}
The \emph{rot} of a permutation $x\in S_n$ is $(234\cdots n1)\circ x$. Intuitively, $\operatorname*{rot}(x)$ can be obtained by taking each letter $j$ in $x$ and replacing it with the letter in $[1,n]$ congruent to $j+1$ modulo $n$.   
\end{defn}

\begin{ex} If we apply rot to $32451$, we get $\operatorname*{rot}(32451) = 43512$.
\end{ex}

Now we can use \emph{rot} to compare class sizes when $n$ is even.

\begin{thm}\label{even n rot} For even $n$, there are the same number of odd and even non $K$-avoiders in $S_n$. 
\end{thm}
\begin{proof}
Because $m = \id_c$, the operation $\operatorname*{rot}$ maintains the
position and the existence of $K$-hits in a permutation even though
the letters of the $K$-hits are changed. Therefore, a bijection from
the even to the odd non $K$-avoiders in $S_n$ can be constructed by
mapping each even non $K$-avoider $\omega$ to $\operatorname*{rot}
\left( \omega \right)$. Indeed, because $n$ is even,
$\operatorname*{rot}$ changes parity.
\end{proof}

The case where $n$ is odd is much more interesting. To compare the
sizes of the two nontrivial classes, we want to compare the number of
even and odd non $K$-avoiders in $S_n$. Restricting ourselves to $c >
\frac{n}{2}$, we will go through a series of problem transformations,
ultimately bringing us to an enumeration using the Catalan number. In
particular, we will transform the problem from being
\begin{center}
about non $K$-avoiders \\
 $ \downarrow $ \\
about hit-huggers \\
 $ \downarrow $\\
about Latin reading words of $2 \times (n-c-1)$ standard Young tableaux \\
 $ \downarrow $\\
about areas above Catalan paths in a $(n-c-1)\times (n-c-1)$ grid.

\vspace{1 cm}
\end{center}

Lemma \ref{hit-ended --> non-avoiding} will perform the first
transformation, and will require the following notation.
\begin{defn}
We will denote the sets of hit-huggers, non $K$-avoiders, and
hit-openers in $S_j$ as $H_j, N_j,$ and $P_j$, respectively. A
superscript of ``odd'' or ``even'' on the right restricts the set to
elements of that parity. A subscript (respectively superscript) on the
left restricts the set to elements that begin (respectively end) with
the letter in the subscript (respectively superscript). For example,
$_1N_{n-1}^{even}$ refers to the set of even non $K$-avoiders that
begin with $1$ in $S_{n-1}$.
\end{defn}

 \begin{lem}\label{hit-ended --> non-avoiding} If $n$ is odd and $c>\frac{n}{2}$, then $|^nH_n^{even}| - |^nH_n^{odd}| = |^nN_n^{odd}| - |^nN_n^{even}|$. In other words, (the number of even hit-huggers in $S_n$ ending with $n$) $-$ (the number of odd hit-huggers in $S_n$ ending with $n$) equals (the number of odd non $K$-avoiders ending with $n$) $-$ (the number of even non $K$-avoiders ending with $n$) in $S_n$.
\end{lem}
\begin{proof}
Every non $K$-avoider $\psi$ in $S_n$ starting with $1$ is either a hit-opener or has the form $1 \rightharpoonup \pi$ for some non $K$-avoider $\pi \in S_{n-1}$. Since $1 \rightharpoonup \sigma$ has the same parity as $\sigma$ for every $\sigma \in S_{n-1}$, this yields the two equations,

\begin{align*}
|_1N_n^{odd}|& = |N_{n-1}^{odd}| + |_1P_n^{odd}|, \numberthis \label{huggereq1} \\
|_1N_n^{even}|& = |N_{n-1}^{even}| + |_1P_n^{even}|. \numberthis \label{huggereq2}
\end{align*}

Recalling by Theorem \ref{even n rot} that $|N_{n-1}^{odd}| - |N_{n-1}^{even}| = 0$, we may subtract Equation \ref{huggereq2} from Equation \ref{huggereq1} to get
\begin{align*}
|_1N_n^{odd}| - |_1N_n^{even}| = |N_{n-1}^{odd}| + |_1P_n^{odd}| - |N_{n-1}^{even}| - |_1P_n^{even}| = |_1P_n^{odd}| - |_1P_n^{even}|. \numberthis \label{huggereq3}
\end{align*}

Before using Equation \ref{huggereq3}, we must first derive Equation \ref{huggereq7}. For every hit-opener $\sigma$ in $S_{n-1}$, we know that $\sigma \leftharpoonup n$ is either a hit-opener or a hit-hugger in $S_n$. Each hit-opener and hit-hugger in $S_n$ ending in $n$ can be written as $\sigma \leftharpoonup n$ for exactly one such $\sigma$.  Since $\sigma \leftharpoonup n$ has the same parity as $\sigma$, this results in the two equations,
\begin{align*} 
|P_{n-1}^{odd}| &= |^nP_n^{odd}| + |^nH_n^{odd}|, \numberthis \label{huggereq4} \\ 
|P_{n-1}^{even}| &= |^nP_n^{even}| + |^nH_n^{even}|. \numberthis \label{huggereq5} 
\end{align*}

Subtracting Equation \ref{huggereq4} from Equation \ref{huggereq5}, we get we get
\begin{align*}
|P_{n-1}^{even}| - |P_{n-1}^{odd}| = |^nP_n^{even}| + |^nH_n^{even}| - |^nP_n^{odd}| - |^nH_n^{odd}|. \numberthis \label{huggereq6} 
\end{align*}

To simplify this equation, we observe that $\operatorname*{rot}$ bijects even and odd hit-openers in $S_{n-1}$. (See the proof of Lemma \ref{even n rot} for discussion of the bijection.) Thus the left side of Equation \ref{huggereq6} becomes zero, and we get
\begin{align*}
|^nP_n^{odd}| - |^nP_n^{even}| = |^nH_n^{even}| -  |^nH_n^{odd}|. \numberthis \label{huggereq7} 
\end{align*}

Before concluding the proof, we need two more equations. Note that $\operatorname*{rot}$ maintains parity in $S_n$ because $n$ is odd. Thus we have
\begin{align*}
|^nN_n^{odd}| - |^nN_n^{even}| &= \frac{1}{n} \cdot |N_n^{odd}| - \frac{1}{n} \cdot |N_n^{even}| = |_1N_n^{odd}| - |_1N_n^{even}|, \text{ and similarly, } \numberthis \label{huggereq8} \\
|_1P_n^{odd}| - |_1P_n^{even}| &= |^nP_n^{odd}| - |^nP_n^{even}|. \numberthis \label{huggereq9}
\end{align*}

Finally, plugging in Equation \ref{huggereq8}, followed by Equation \ref{huggereq3}, followed by Equation \ref{huggereq9}, followed by Equation \ref{huggereq7}, we get by transitivity of equality that
$$|^nN_n^{odd}| - |^nN_n^{even}| = |^nH_n^{even}| - |^nH_n^{odd}|.$$

\end{proof}

Next we transform the problem from being about hit-huggers to being about Latin reading words of Young tableaux. 

\begin{conv}
All Young tableaux mentioned are Standard Young tableaux in the English notation.
\end{conv}

\begin{lem}\label{hit-ended-->Young tableaux} Let $n$ be odd and  $c>\frac{n}{2}$. There is a bijection between hit-huggers in $S_n$ ending in $n$ and $2 \times (n-c-1)$ Young tableaux. The parity of the Latin reading word (the word obtained by reading the entries of a Young tableau from left to right like a paragraph) of a $2 \times (n-c-1)$ Young tableau is the same as that of its corresponding hit-hugger in $S_n$.  
\end{lem}

\begin{proof}
Let $a$ be an arbitrary hit-hugger ending with $n$. We have $$a = a_1a_2a_3\cdots a_{n-c}a_{n-c+1} \cdots a_c a_{c+1}a_{c+2}\cdots a_{n-1}a_n.$$ The last $c$ letters of $a$ strictly increase from left to right since $a_n = n$ is the largest letter, so 
\begin{align*}
a_{n-c+1}< \cdots <a_c <a_{c+1} <a_{c+2}< \cdots <a_{n-1} <a_n. \numberthis \label{letterorder1}
\end{align*} 
In order for letters $a_{n-c}$ through $a_{n-1}$ to be $K$-avoiding, we must have $a_{n-c} > a_{n-c+1}$. However, we know that, $a_1a_2a_3\cdots a_c$ has to be a $K$-hit, implying that 
\begin{align*}
a_{n-c+1}<a_{n-c+2}<\cdots <a_c<a_1<a_2<a_3<\ldots <a_{n-c}. \numberthis \label{letterorder2}
\end{align*}

In order for $a_2a_3\cdots a_{n-c}a_{n-c+1} \cdots a_ca_{c+1}$ to be $K$-avoiding, we must have $a_2<a_{c+1}$. Similarly, we get all of the inequalities,
\begin{align*}
a_2 < a_{c+1}, \ \ a_3<a_{c+2}, \ \  a_4<a_{c+3}, \ \ \ldots , \ \ a_{n-c}<a_{n-1}. \numberthis \label{letterorder3} 
\end{align*}

Observe that Equations \ref{letterorder1}, \ref{letterorder2}, and \ref{letterorder3} are sufficient to prove that $a$ is a hit-hugger. Indeed, Equations \ref{letterorder1} and \ref{letterorder2} imply that $a$ contains a hit beginning at its first letter and a hit ending at its final letter. Then Equation \ref{letterorder3} insures that $a$ contains no additional hits. Thus for $a$ to be a hit-hugger ending with $n$ is exactly equivalent to $a$ satisfying the three equations. 

If we combine together Equations \ref{letterorder1}, \ref{letterorder2}, and \ref{letterorder3}, we get

\begin{align*}
a_{n-c+1} &<& \cdots &<& a_c &<& a_1  &< & a_2   & < &a_3   & < &\cdots &<&  a_{n-c} & \\
&&&&&&&&\wedge    &  & \wedge   &  &  \wedge      & &  \wedge & \\ 
&& a_{n-c+1} &<& \cdots &<& a_c &<  &a_{c+1}& < &a_{c+2} &<& \cdots &<&   a_{n-1} &< a_n.
\end{align*}

Observe that the values for $a_{n-c+1}, \ldots, a_c$ and for $a_1$ are uniquely determined. Indeed, they are less than the other letters and are strictly ordered with respect to each other. Noting that $a_n = n$, the values of the not yet determined letters are constrained only by

\begin{align*}
a_2       & < & a_3      & < &      \cdots &<&  a_{n-c}  \\
\wedge    &   & \wedge   &  &  \wedge      & &  \wedge \\ 
a_{c+1}    & < & a_{c+2}    &<&        \cdots &<&   a_{n-1}.
\end{align*}

Thus we have found our desired bijection between hit-huggers ending with $n$ and $2 \times (n-c-1)$ Young tableaux.

It remains to show the claim about parity. Since $\left(n-c\right)\left(2c-n\right)$ is even, sliding the subsequence $a_{n-c+1}\cdots a_{c}$ to be at the start of $a$ will yield a permutation of the same parity as $a$. In this new permutation, \\ $a_{n-c+1}\cdots a_ca_1a_2a_3\cdots a_{n-c}a_{c+1}a_{c+2}\cdots a_{n-1}a_n$, all inversions are within the letters $a_2$ through $a_{n-1}$, since $a_{n-c+1}\cdots a_{c}$ are the smallest letters in increasing order and $a_n$ is the largest letter. Furthermore, the contiguous subword from $a_2$ to $a_{n-1}$ is exactly the Latin reading word of the Young tableau corresponding to $a$. Thus, the parity of the Latin reading word of a $2 \times (n-c-1)$ Young tableau is the same as that of its corresponding hit-hugger in $S_n$.
\end{proof}

Finally, we transform the problem to being about areas above Catalan paths on a grid.

\begin{defn}
A \emph{Catalan path} on a $j \times j$ grid is a path from the top left corner to the grid to the bottom right corner comprising only steps to the right and steps down. Moreover, it is restricted to stay on or above the main diagonal at all times.
\end{defn}

\begin{defn}
The \emph{area} of a Catalan path is the area above it in the grid. A \emph{odd path} is a Catalan path with odd area above it. An \emph{even path} is one  with even area above it.
\end{defn}

\begin{ex} Figure \ref{exampleeven} shows an example of an even path.
\end{ex}

Let $C_k$ denote the number of Catalan paths on a $k \times k$ grid. Recall that $C_k = \frac{1}{k+1}{2k\choose k}$ is the $k$th Catalan number.

\begin{figure}
\includegraphics[scale=0.43]{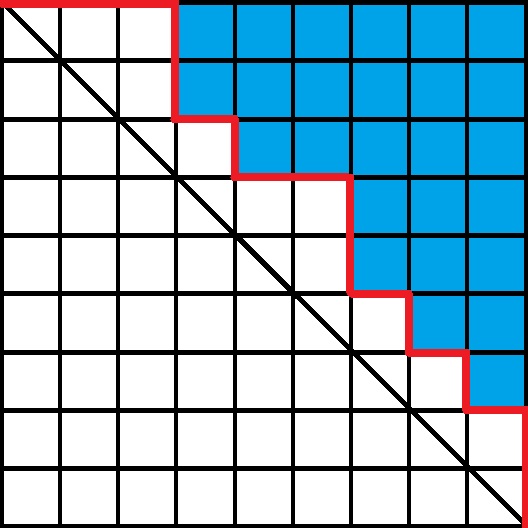} \\
\caption{An example Catalan path on a $9 \times 9$ grid. The shaded region is the area. Because the area is 26, the path is even.} 
\label{exampleeven}
\end{figure}

The final transformation of the problem:
\begin{lem}\label{tableaux to lattice paths} For odd $n$ there is a bijection between $2 \times (n-c-1)$ Young tableaux and Catalan paths inside a $(n-c-1)\times (n-c-1)$ grid. The parity of the area above the Catalan path equals the parity of the Latin reading word of the corresponding Young tableau. 
\end{lem}
\begin{proof}
We use a standard bijective map from $(n-c-1) \times (n-c-1)$ Catalan paths to $2 \times (n-c-1)$ Young tableaux, an example of which appears in Figure \ref{examplebijection}. Given any $(n-c-1) \times (n-c-1)$ Catalan path, traverse the steps in the path from its beginning to end. At the $i$th step, add $i$ to the tableaux we are building as follows. If the $i$th step is a down step, add it to the bottom row. Otherwise, add it to the top row. The inverse of the bijection is easy to construct. Given any $2 \times (n-c-1)$ Young tableau, start in the top left corner of the grid. Traverse the Young tableau in order of increasing entries. For every number in the traversal: if it is in the top row, draw a step horizontally to the right in the grid; if it is in the bottom row, draw a step vertically down. At any given point in the resulting path, the number of down steps is bounded by the number of steps to the right so far in the path. Thus the resulting path is Catalan.

We will now look at parity, showing that the area of a Catalan path is the same as the inversion number of the Latin reading word of the corresponding Young tableau. Let $t$ be a $2 \times (n-c-1)$ Young tableau. For every pair of numbers $i<j$ where $i$ is in the bottom row of $t$ and $j$ is in the top row of $t$, there is an inversion in the Latin reading word of $t$. (These are exactly the inversions.) At the same time, the Catalan path corresponding to $t$ has a unit square of area in the same row and column as the pair of steps that correspond to $i$ and $j$ (because the step corresponding with $i$ is vertical, the step corresponding with $j$ is horizontal, and the first appears before the second in the path). It is straightforward to see that this is a bijection between squares in the area region of the Catalan path and inversions in the Latin reading word of $t$.
\end{proof}

\begin{figure}
\begin{tabular}{m{7 cm} m{2 cm} m{5 cm}}
\large
\begin{tabular}{l l l l l l l l l}
1 &2 &3 &6 &8 &9 &12 &14 &16 \\
4 &5 &7 &10 &11 &13 &15 &17 &18
\end{tabular}
\hspace{-1 cm}
& \LARGE $\longleftrightarrow$  

& \includegraphics[scale=0.35]{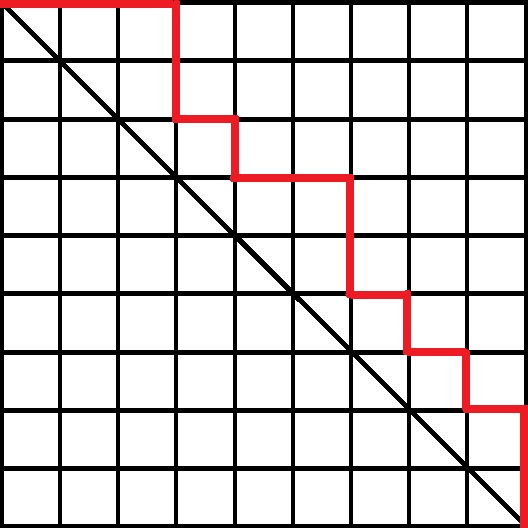} \\
\end{tabular}
\caption{Example of bijection from $2 \times 9$ Young tableaux to $9 \times 9$ Catalan paths.}
\label{examplebijection}
\end{figure}

\begin{defn}
A \emph{leg} of a Catalan path is a maximal sequence of consecutive steps in the same direction.
\end{defn}

\begin{ex} In Figure \ref{exampleeven}, the legs of the path are of lengths 3, 2, 1, 1, 2, 2, 1, 1, 1, 1, 1, and 2 respectively.
\end{ex}

With the problem transformations complete, we obtain an enumerative result:
\begin{lem}\label{diff even odd lattice paths} For odd $n$, there are $C_{\frac{n-c-2}{2}}$ more lattice paths inside a $(n-c-1)\times (n-c-1)$ grid with even than with odd area above the path.
\end{lem}
\begin{figure}
\begin{center}
\begin{tabular}{ccc}
\includegraphics[scale=0.25]{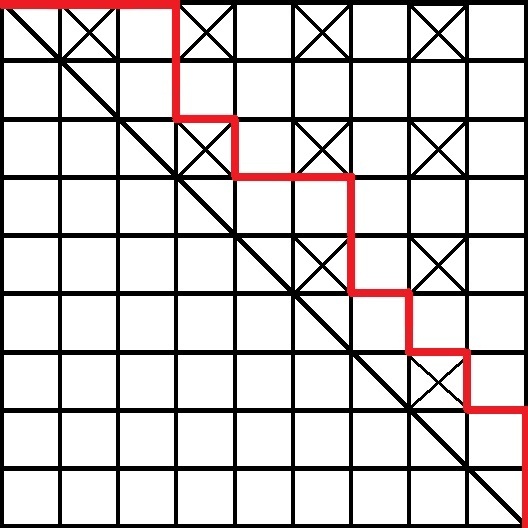} &
\includegraphics[scale=0.25]{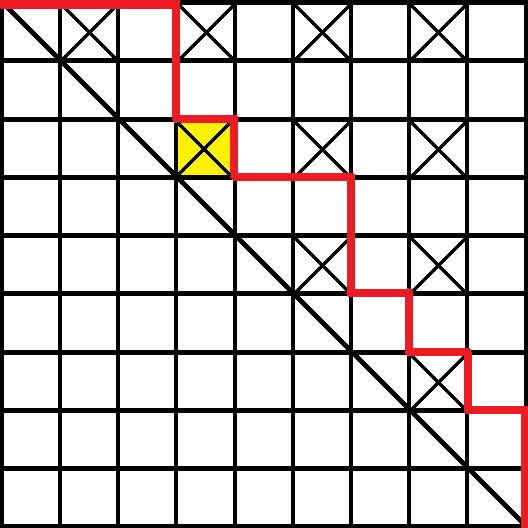} &
\includegraphics[scale=0.25]{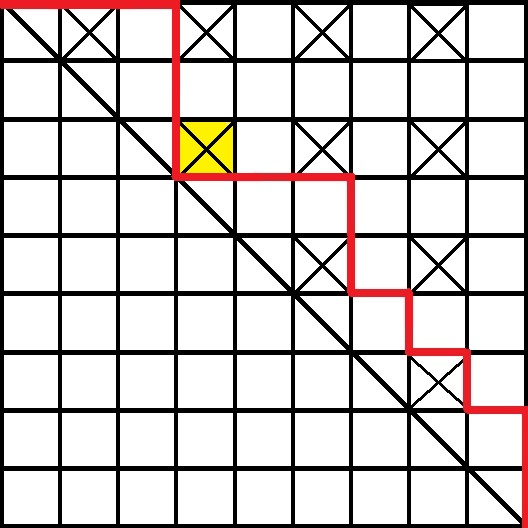} \\
(a) & (b) & (c) \\
\includegraphics[scale=0.25]{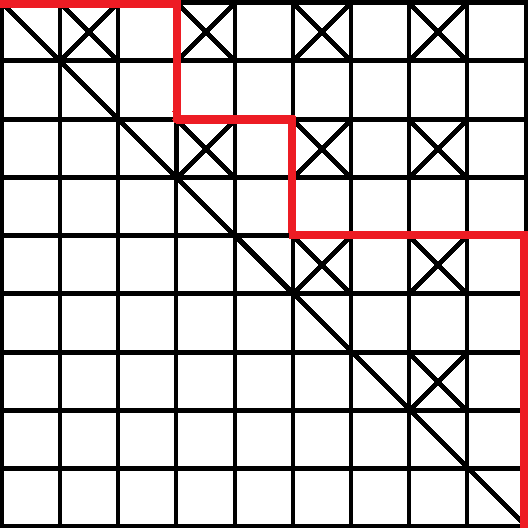} &
\includegraphics[scale=0.25]{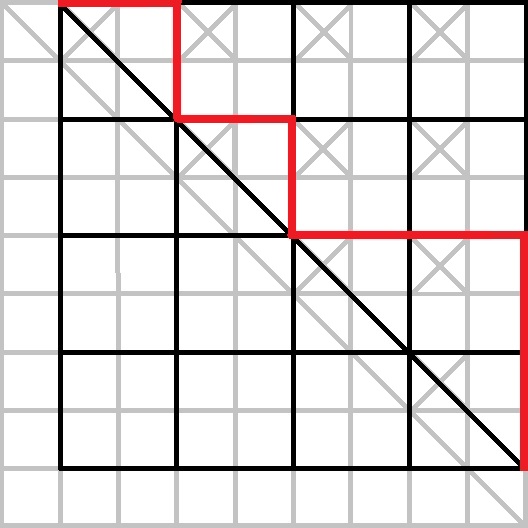} & \\
(d) & (e) & \\
\end{tabular}
\end{center}
\caption{Lattice paths in a grid demonstrating the technique in the proof of Lemma \ref{diff even odd lattice paths}.}
\label{latticepaths}
\end{figure}


\begin{proof}
Given a Catalan path $l$ on a $(n-c-1)\times (n-c-1)$ grid, draw X's
inside the squares that are in even columns and odd rows. An example
is shown in Figure~\ref{latticepaths}(a).

Take the first time the path touches two edges of the same square with
an X in it (highlighted in Figure~\ref{latticepaths}(b)) and draw the
path $l'$ that is the same as $l$ except that it touches the other two
edges of that square (shown in Figure~\ref{latticepaths}(c)). Note
that $l'$ also gets mapped to $l$, and that the area above $l$ has the
opposite parity as the area above $l'$.


Thus, there are the same number of even and odd Catalan paths that
touch at least one square with an X on two edges. Consider all of the
paths that do not touch any squares with an X on two edges (for
example, Figure~\ref{latticepaths}(d)). Notice that these paths must
not have any odd-length legs except for the first and last legs and
that there is an even area above each such path.

This means that the number of these paths is equal to the difference
between the number of odd and even paths. By deleting the first column
and last row, and shrinking the dimensions of the grid each by a
factor of two (shown in Figure~\ref{latticepaths}(e)), we see that we
only have to count the number of Catalan paths inside a
$\frac{n-c-2}{2}\times \frac{n-c-2}{2}$ grid, which is just
$C_{\frac{n-c-2}{2}}$.
\end{proof}

\begin{rem}
Note that a result equivalent to Lemma \ref{diff even odd lattice
  paths} was proven by \cite{SS}. However, we choose to provide our
proof because it is very different than the previous one.
\end{rem}


We now restate Lemma \ref{diff even odd lattice paths} in the context
of our original problem.
\begin{thm}
\label{difference in parity of odd n} Let $n$ and $c$ be odd, $c>\frac{n}{2}$, and $m = \id_c$. Let $K$ contain the cyclic shifts of $m$. Then $S_n$ has two nontrivial classes under $K$-equivalence. Moreover, the equivalence class containing odd permutations contains $nC_{\frac{n-c-2}{2}}$ more permutations than the the one containing even permutations. 
\end{thm}
\begin{proof}
 By Lemma \ref{hit-ended --> non-avoiding}, $|^nH_n^{even}| - |^nH_n^{odd}| = |^nN_n^{odd}| - |^nN_n^{even}|$. Therefore, due to $\operatorname*{rot}$ preserving the parity of a permutation, it suffices to show that there are $C_{\frac{n-c-2}{2}}$ more even than odd hit-huggers ending with the letter $n$. By Lemma \ref{hit-ended-->Young tableaux}, the parity of the Latin reading word of a $2 \times (n-c-1)$ Young tableau is the same as that of a corresponding hit-hugger ending with $n$ in $S_n$. By Lemma \ref{tableaux to lattice paths}, this is the same as the parity of the area above a corresponding Catalan path inside a $(n-c-1)\times (n-c-1)$ grid. Since by Lemma \ref{diff even odd lattice paths}, there are $C_{\frac{n-c-2}{2}}$ more even than odd area-ed Catalan paths inside a $(n-c-1)\times (n-c-1)$ grid, we are done.
\end{proof}

\section{$\mathbf{\{ x \mid x\in S_{c+1},\ x_1=1 \}}$-Equivalence}\label{subsecstartwith1}

Until now, we have used $c$ to denote the size of the permutations in
the replacement partition we are considering. We now diverge from this
convention and assign $c$ to a value of one smaller. Although this may
seem like a strange choice, it will simplify formulas throughout.

In this section, we study the pattern-replacement equivalence
determined by the replacement set comprising every permutation in
$S_{c+1}$ which starts with $1$. For example, when $c=2$, this is the
$\{123, 132\}$-equivalence, which \cite{LPRW} previously
considered but was unable to enumerate. We note, however, that
\cite{LPRW} was able to count the size of the class containing the
identity under $\{123, 132\}$-equivalence, a result which we extend to
the general case in Corollary \ref{coridentityclass1} and Theorem
\ref{thmcountsquished}.

In this section, we will sometimes refer to the \emph{direct
  product} of the equivalence class of a word $x$ with the
equivalence class of a word $y$.
\begin{defn}
The \emph{direct product} of the equivalence class of a word $x$
with the equivalence class of a word $y$ is the set of
concatenations of elements from the first with elements from the
second.
\end{defn}

For a given $c$, we count the equivalence classes in $S_n$ under
$\{ x \mid x\in S_{c+1},\ x_1=1 \}$-equivalence (Theorem
\ref{thmgetf}). In the process, we characterize the equivalence class
containing the identity permutation (Corollary
\ref{coridentityclass1}) and find a formula for its size (Theorem
\ref{thmcountsquished}). Then, for a given $w \in S_n$, we
characterize its equivalence class as the direct product of
equivalence classes containing the identity in $S_j$ for various $j$
(Remark \ref{remcount1}). Combining our results together, we have a
recipe for characterizing and finding the size of any equivalence
class (Remark \ref{remcount1}).

We begin by defining a special property of permutations.

\begin{defn}
We call a permutation $k$\emph{-squished} if for each $1 \le j \le k$,
$j$ is in the left-most $c(j-1)+1$ letters.
\end{defn}

Observe for $k \ge 1$, every $k$-squished permutation contains 1 in
the first position.

The importance of $k$-squished-ness will stem from the fact that it is
an invariant:
\begin{lem}\label{lemsquishclosure}
Let $x, y \in S_n$ such that $x \equiv y$ under $\{x \mid x \in
S_{c+1}, \ x_1 = 1\}$-equivalence. If $x$ is $k$-squished for a given
$k$, then so is $y$.
\end{lem} 
\begin{proof}
This is straightforward.
\end{proof}

It turns out that we can characterize the permutations equivalent to
the identity using $k$-squished-ness. This extends a result of
\cite{LPRW} for the case of $c=2$.
\begin{lem}\label{squishinglem}
Let $w$ be a $k$-squished permutation where $c+1 \le n \le
ck+1$. Then, $w$ is equivalent to the identity under $\{ x \mid x\in
S_{c+1},\ x_1=1 \}$-equivalence.
\end{lem}
\begin{proof}
Note that since $w$ is $k$-squished with $ck + 1 \ge n$, $w$ is
actually $n$-squished. In particular, for $j > k$, the letter $j$ in
$w$ appears in position at most $ck + 1$ and thus is position at most
$c(j - 1) + 1$.

Assume that the lemma holds in $S_{n-1}$ (with a trivial inductive
base case of $S_{c+1}$). Let $w$ be a $n$-squished permutation of
length $n \ge c + 2$. Let $r$ be $\lceil (n-1)/c \rceil$. Because $w$
is $n$-squished, $r$ is in position at most $(\lceil (n-1)/c \rceil -
1)c + 1 < n$, and thus cannot be in the final position of
$w$. Applying the inductive hypothesis to the first $n-1$ letters of
$w$, we can reach a permutation $w'$ which has each letter $j \le r$
in position $c(j-1)+1$; and which with exception of the final letter,
has the other letters in increasing order. In $w'$, $r$ is in position
$(\lceil (n-1)/c \rceil - 1) c+1 \ge n - c$ and is thus in the final
$c+1$ positions but not in the final position; $r-1$ is $c$ positions
before it. Using the $\{x \mid x \in S_{c+1}, \ x_1 = 1\}$-hit
starting with $r-1$, we rearrange $w'$ to have $r$ in position $n-c$,
while not changing the relative order of any other letters. At this
point, $r$ is followed by the final $c$ letters in the permutation,
each of which exceeds $r$ in value. Thus the final $c+1$ letters of
the permutation form a $\{x \mid x \in S_{c+1}, \ x_1 =
1\}$-hit. Noting that $n$ is in the final two positions of the
permutation, we can rearrange the final two letters of the permutation
so that $n$ is in the final position. Finally, we may apply the
inductive hypothesis to the first $n-1$ letters and reach the
identity.
\end{proof}

\begin{cor}\label{coridentityclass1}
Suppose $c+1 \le n$. Then the permutations in $S_n$ equivalent to the
identity under $\{x \mid x \in S_{c+1}, \ x_1 = 1\}$-equivalence are
exactly those which are $k$-squished for all $k \leq n$.
\end{cor}
\begin{proof}
The identity permutation is $k$-squished for all $k \leq n$. It
follows by Lemma \ref{lemsquishclosure} that the same is true for
every equivalent permutation under $\{x \mid x \in S_{c+1}, \ x_1 =
1\}$-equivalence. On the other hand, suppose $w$ is $k$-squished for
all $k \leq n$. Then Lemma \ref{squishinglem} implies $w$ is equivalent
to the identity under $\{x \mid x \in S_{c+1}, \ x_1 =
1\}$-equivalence.
\end{proof}

It is interesting to compute the size of the equivalence class
containing the identity.
\begin{thm}\label{thmcountsquished}
Let $n \ge c+1$. Let $r = \lfloor (n-1) / c \rfloor + 1$. The size of
the equivalence class containing the identity in $S_n$ under $\{x
\mid x \in S_n, \ x_1 = 1\}$-equivalence is
$$(n-r)! \prod_{j \in [1, r]} \left((j-1)(c-1) + 1\right).$$
\end{thm}
\begin{proof}
By Corollary \ref{coridentityclass1}, it suffices to count the
$n$-squished permutations in $S_n$. The letters in the set $[1,
  \lfloor (n-1) / c \rfloor + 1]$ are constrained by this. The letter
$1$ has only one position in which it can be placed. Once $1$ has
been placed, there are $c$ positions in which $2$ can be placed (so
that it appears in the first $c + 1$ positions). Then there are $2c -
1$ positions in which $3$ can be placed (so that it appears in the
first $2c + 1$ positions). Continuing like this, if $r = \lfloor (n-1)
/ c \rfloor + 1$, then there are
$$\prod_{j \in [1, r]} \left((j-1)(c-1) + 1\right)$$ ways to choose
the positions for the letters $1, \ldots, r$. The remaining $n - r$
letters can be placed in the remaining positions arbitrarily. The
total number of $n$-squished permutations is therefore
$$(n-r)! \prod_{j \in [1, r]} \left((j-1)(c-1) + 1\right).$$
\end{proof}

Let $f(n)$ be the number of classes in $S_n$ under $\{ x \mid x\in
S_{c+1},\ x_1=1 \}$-equivalence. We will ultimately find a formula for
$f(n)$ in Theorem \ref{thmgetf}.

In order to do so, we will first need an additional function. Let
$g(k, n)$ be the number of classes in $S_n$ containing permutations
that are $k$-squished under $\{ x \mid x\in S_{c+1},\ x_1=1
\}$-equivalence.

We start by computing $g$:
\begin{lem}\label{lemcountsquishedclasses}
$$g(k, n) =
\begin{cases}
 (n-1)! , & \text{if } n<c+1 \\
 1, & \text{if } c+1 \le n \le ck+1 \\
g(k+1, n)+\sum\limits_{j=ck + 2}^{n}{g(k, j-1) \cdot g(1,n-j+1) \cdot {n-k-1 \choose n-j}}, & \text{ if } n > ck+1
\end{cases}
$$

\end{lem}
\begin{proof}

\textbf{\center Case of $c+1 > n$:} In this case, are $(n-1)!$ $k$-squished permutations
in $S_n$ because they are exactly the permutations beginning with
$1$. No transformations can be applied to a permutation of size less
than $c+1$. Thus each $k$-squished permutation is in its own
equivalence class.

\textbf{\center Case of $c+1 \leq n$ and $ck + 1 \ge n$: } In this case, the
$k$-squished permutations are all equivalent to the identity (Lemma
\ref{squishinglem}). This was also shown by \cite{LPRW} for the case
of $c=2$.

\textbf{\center Case of $ck + 1 < n$: } Consider the permutations in $S_n$
which are $k+1$-squished. By definition, such permutations fall into
$g(k+1, n)$ classes.

Let $w \in S_n$ be $k$-squished but not $k+1$-squished. Let $i$ be the
position of $k+1$ in $w$. Because $w$ is not $k+1$-squished, we know
that $i \ge ck + 2$. On the other hand, for all $x \in S_n$ such that
$x \equiv w$, we know that $x$ is $k$-squished (Lemma
\ref{lemsquishclosure}). It follows that every letter of value $\leq
k$ is in position $< c(k-1) + 1$ in $x$. So using $\{x \mid x \in S_n,
\ x_1 = 1\}$-rearrangements, we can never push any letter less than
$k+1$ far enough to the right so that it will be in a $\{x \mid x \in
S_{c+1}, \ x_1 = 1\}$-hit with $k+1$. Consequently, $k+1$ will never
be in a $\{x \mid x \in S_{c+1}, \ x_1 = 1\}$-hit where it is not the
first letter.

It follows that the equivalence class containing $w$ in $S_n$ is the
direct sum of the equivalence class containing the first $i-1$ letters
of $w$ with the equivalence class containing the final $n-i+1$ letters
of $w$.

There are $g(k, i-1)$ possibilities for the equivalence class on the
left and $g(1, n-i+1)$ possibilities for the equivalence class on the
right.

There are ${n-k-1 \choose n-i}$ options for which letters precede
position $i$ and which letters follow position $i$ in $w$. In
particular, we do not get any choice for the letters
$1,\ldots,k+1$. But we do get to choose which $n-i$ of the remaining
$n-k-1$ letters are to be to the right of the $i$-th position.

Thus the number of equivalence classes containing permutations in
$S_n$ which are $k$-squished but not $k+1$-squished, and which contain
$k+1$ in the $i$-th position, is $${n-k-1 \choose n-i}g(k, i-1)g(1,
n-i+1).$$

Summing over the possibilities for $i$, we find that the number of
classes in $S_n$ containing $k$-squished but not $k+1$-squished
permutations is
$$\sum\limits_{j=ck + 2}^{n}{g(k, j-1) \cdot g(1,n-j+1) \cdot {n-k-1
    \choose n-j}},$$
completing the proof.
\end{proof}

Now we can count all of the equivalence classes in $S_n$.
\begin{thm}\label{thmgetf}
Recall that $f(n)$ is the number of equivalence classes in $S_n$ under $\{x \mid x \in S_n, \ x_1 = 1\}$-equivalence. Then
$$
 f(n) =
\begin{cases}
 n! , & \text{if }n<c+1 \\
\sum\limits_{j=1}^{n}{f(j-1)\cdot g(1,n-j+1) \cdot {n-1 \choose j-1}}, & \text{if }n\ge c+1
\end{cases}
$$
\end{thm}
\begin{proof}
As before, the case of $n<c+1$ is trivial. Consider the case of $n \ge
c+1$. Under the transformations considered, the position of $1$ does
not change. Furthermore, the letters to the left and right of $1$ move
independently of each other. In other words, if $w \in S_n$ contains
$1$ in position $j$, then the equivalence class containing $w$ is the
direct sum of the equivalence class containing $w_1w_2\cdots w_{j-1}$
with the equivalence class containing $w_jw_{j+1}\cdots w_n$.

Thus permutations with $1$ in position $j$ fall into $f(j-1)\cdot g(1,
n-j+1)$ classes for a given choice of which letters are to the left
and right of $1$. There are $n-1 \choose j-1$ such choices, resulting
in $f(j-1)\cdot g(1, n-j+1) \cdot {n-1 \choose j-1}$ classes. If we
sum for each possible $j$, we get $$\sum\limits_{j=1}^{n}{f(j-1)\cdot
  g(1,n-j+1) \cdot {n-1 \choose j-1}} \text{ classes in }S_n.$$
\end{proof}

In the preceding results, we have already implicitly characterized the
equivalence classes in $S_n$ under $\{x \mid x \in S_n, \ x_1 =
1\}$-equivalence. We formalize this in the following remark.

\begin{rem}\label{remcount1}
Given a permutation $w$ in $S_n$ (where $n \ge c+1$), one can solve
for its class size under $\{x \mid x \in S_n, \ x_1 =
1\}$-equivalence (and characterize the class along the way).

If $w$ is $n$-squished, then the enumeration is simply the number of
$n$-squished permutations in $S_n$ (by Corollary
\ref{coridentityclass1}). This, in turn comes from Theorem
\ref{thmcountsquished}.

Otherwise, we take the smallest positive $k$ such that $w$ is not
$k$-squished. Let $i$ be the position of $k$ in $w$. In the proof of
Lemma \ref{lemcountsquishedclasses}, we established that if $k>1$ then
the equivalence class containing $w$ in $S_n$ is the direct sum of the
equivalence class containing the first $i-1$ letters of $w$ with the
equivalence class containing the final $n-i+1$ letters of $w$. In the
proof of Theorem \ref{thmgetf} we showed this to also be true when
$k=1$. Thus we can recursively compute the sizes of the two
equivalence classes in the direct product to get the size of the
equivalence class containing $w$.
\end{rem}

\begin{rem}
Observe from Remark \ref{remcount1} that every equivalence class has a
size which is either 1 or is the product of values obtained in various
cases of Theorem \ref{thmcountsquished}.
\end{rem}

\section{$\mathbf{ \{ x \mid x\in S_{c+1},\ x_1=1 \} \ \{ x \mid x\in S_{c+1},\ x_{c+1}=1 \}}$-Equivalence}\label{subsecstartandendwith1}

As in the previous section, we define $c$ to be one less than the
size of the permutations in the replacement partition under
consideration.

In this section, we study the $\{ x \mid x\in
  S_{c+1},\ x_1=1 \} \ \{ x \mid x\in S_{c+1},\ x_{c+1}=1
  \}$-equivalence. For example, when $c=2$, this is the $\{123,
132\}\{231, 321\}$-equivalence. 

This section proceeds as follows. For a given $c$, we count the
equivalence classes in $S_n$ under $\{ x \mid x\in
  S_{c+1},\ x_1=1 \}$-equivalence (Theorem \ref{thmgetf}). In the
process, we characterize and enumerate certain equivalence classes,
including the class containing the identity permutation (Lemma
\ref{lemsquish2} and Theorem \ref{thmcountsquished2}). Then, for a
given $w \in S_n$, we characterize its equivalence class as the
truncated direct product of equivalence classes of the form already
characterized (Remark \ref{remcount2}). Combining our results
together, we have a recipe for characterizing and finding the size of
any equivalence class (Remark \ref{remcount2}).

Our approach throughout this section is an extension of our
approach in Section \ref{subsecstartwith1}. In fact, we will find it useful to at
times explicitly reference Lemma \ref{squishinglem} from the previous
section.

Before continuing, we remark that Theorem \ref{thmgetf2} can be
used to count the equivalence classes under $\{123, 132\}\{231,
321\}$-equivalence. If we combine this enumeration with that of the
forgotten relation \cite{NS}, that of the Knuth relation \cite{Kn},
and those found by Kuszmaul \cite{Kuszmaul12}, then there is only one
remaining unenumerated replacement partition of the form
$\{a,b\}\{c,d\}$ where $a,b,c,d \in S_3$. In particular, $\{213,
312\}\{132, 231\}$-equivalence is the only remaining such equivalence
for which there is no known formula counting equivalence classes.

\begin{defn}
Let $j,k,n$ be nonnegative integers such that $j+k \le n$. Let $w \in
S_n$. We say that $w \in S_n$ is \emph{$j,k$-squished} if $w$ is the
concatenation of two words $a$ and $b$ such that $a$ forms a
$j$-squished permutation and $b$ forms a $k$-squished permutation
written from right to left. (i.e., if we write $b$ backwards, it forms
a $k$-squished permutation.)

Let $J$ be the smallest $j$ letters of $a$ and $K$ be the smallest $k$
letters of $b$ (in value). Then in addition to $w$ being
$j,k$-squished, we say that $J$ and $K$ \emph{$j,k$-split} $w$.
\end{defn}

The following two lemmas are similar to Lemma \ref{lemsquishclosure}
from the previous section. Note that the second lemma includes the
first in all cases but when $c(j-1)+1 = n$ and $k=0$, or $c(k-1)+1 =
n$ and $j=0$.
\begin{lem}\label{lemclosurerewritten}
Let $x, y \in S_n$ such that $x \equiv y$ under $\{ x \mid x\in
S_{c+1},\ x_1=1 \}$ $\{ x \mid x\in S_{c+1},\ x_{c+1}=1
\}$-equivalence. If $x$ is $k$-squished for a given $k$, then so is
$y$.
\end{lem}
\begin{proof}
This is straightforward
\end{proof}

\begin{lem}\label{lemclosure2}
Suppose $j,k,n$ are nonnegative integers and that $ck + cj + 2 - c \le
n$. Suppose $y \in S_n$ and that $J$ and $K$ are sets which
$j,k$-split $y$. Then for all $z \equiv y$ under $\{ x \mid x\in
S_{c+1},\ x_1=1 \}$ $\{ x \mid x\in S_{c+1},\ x_{c+1}=1 \}$-equivalence,
$J$ and $K$ $j,k$-split $z$.
\end{lem}
\begin{proof}
Because $ck + cj +2 -c \le n$, there is no $\{ x \mid x\in
S_{c+1},\ x_1=1 \}$ $\{ x \mid x\in S_{c+1},\ x_{c+1}=1 \}$-hit in $y$
containing both an element of $J$ and an element of $K$. By Lemma
\ref{lemclosurerewritten}, this remains the case as we perform $\{ x
\mid x\in S_{c+1},\ x_1=1 \}$ $\{ x \mid x\in S_{c+1},\ x_{c+1}=1
\}$-rearrangements in $y$ (since the first $ck + 1$ letters remain
$k$-split ad the final $cj + 1$ letters remain $j$-split). Thus Lemma
\ref{lemclosurerewritten} is sufficient to guarantee that repeated $\{
x \mid x\in S_{c+1},\ x_1=1 \}$ $\{ x \mid x\in S_{c+1},\ x_{c+1}=1
\}$-rearrangements in $y$ yield a permutation which is still
$j,k$-split by $J$ and $K$.
\end{proof}

Let $f(n)$ be the number of classes in $S_n$ under $\{ x \mid x\in
S_{c+1},\ x_1=1 \}$ $\{ x \mid x\in S_{c+1},\ x_{c+1}=1
\}$-equivalence. We will enumerate $f$ in Theorem \ref{thmgetf2}.

When both $ck+cj+2-c \le n$ and $j,k > 0$, or both $jc+1 \ge n$
and $k=0$, or both $kc+1 \ge n$ and $j=0$, then we also define a
function $g(j, k, n)$. The cases for when we define $g$ are motivated
by Lemmas \ref{lemclosurerewritten} and \ref{lemclosure2}. If one of
the three cases holds, then the lemmas guarantee that equivalence
classes containing $j,k$-squished permutations contain only
$j,k$-squished permutations. We define $g(j,k,n)$ to be the number of
classes in $S_n$ comprising permutations that are $j,k$-squished for a
given choice of which $J$ and $K$ $j,k$-split them.

At first sight it may appear that $g(j,k,n)$ is not well-defined
because changing our choice of $J$ and $K$ may change its value. We
will see, however, that this is not the case.

The following is analogous to Lemma \ref{squishinglem} from the previous section:
\begin{lem}\label{lemsquish2}
Let $j,k \ge 1$. Suppose $cj+ck+1 \ge n \ge c+1$ and $ck + cj + 2 - c
\le n$. Then $g(j, k, n)=1$.
\end{lem}
\begin{proof}
Given such $j,k,c,n$, let $w \in S_n$ be $j,k$-squished and let $J$
and $K$ be such that $J$ and $K$ $j,k$-split $w$.

First we show that for any letters $m$ and $n$ in $w$, both greater
than $j+k$ in value, we can swap $m$ and $n$ in $w$ through a series
of $\{ x \mid x\in S_{c+1},\ x_1=1 \}$ $\{ x \mid x\in S_{c+1},\ x_{c+1}=1
\}$-rearrangements. Suppose without loss of generality that $m$
appears in the first $jc$ letters\footnote{Observe that at least one
  of $m$ and $n$ must either appear in the first $jc$ letters or in
  the final $kc$ letters because $cj +ck+1 \ge n$.} of $w$. If $n$
also appears in the first $jc +1$ letters of $w$, then we are done by
applying Lemma \ref{squishinglem} to the first $cj+1$ letters in
$w$. Otherwise, observe that the letter $l$ in position $cj+1$ is
greater than $j+k$ in value (because the first $j+k$ letters in value
are restricted in position by $w$ being $j,k$-squished). Thus we we
may apply Lemma \ref{squishinglem} to the first $cj+1$ letters of
$w$ to swap $m$ and $l$. Then we may apply Lemma \ref{squishinglem} to
the final $ck+1$ letters to swap $m$ and $n$. Finally, we may apply
Lemma \ref{squishinglem} to the first $cj+1$ letters again to this
time swap $l$ and $n$. This completes the swap of $m$ and $n$ in $w$.

As a consequence of being able to swap any such $m$ and $n$, we can
ignore the relative order of such $m$ and $n$ in $w$. In other words,
let $w'$ be $w$ except with that letters with values greater than
$j+k$ having been rearranged to be in increasing order. Then we have
shown that $w \equiv w'$ under $\{ x \mid x\in S_{c+1},\ x_1=1 \}$ $\{
x \mid x\in S_{c+1},\ x_{c+1}=1 \}$-equivalence. Applying Lemma
\ref{squishinglem} to the first $c(j - 1) + 1$ letters of $w'$, we can
slide the letters in $J$ to the front of $w'$ (and to be in increasing
order); and applying the lemma to the final $c(k - 1) + 1$ letters of
$w'$, we can slide the letters in $K$ to the front of $w'$ (and to be
in decreasing order). Thus $w'$ is equivalent to the permutation
beginning with the letters in $J$ in increasing order, ending with the
letters in $K$ in decreasing order, and with all other letters in
increasing order.
\end{proof}

Suppose $cj+ck+1 \ge n \ge c+1$ and $ck + cj + 2 - c \le n$. We have
seen that $g(j, k, n)=1$. The following Theorem gives us the size of
the equivalence class counted by $g(j,k,n)$ in this case.
\begin{thm}\label{thmcountsquished2}
Suppose $cj+ck+1 \ge n \ge c+1$ and $ck + cj + 2 - c \le n$. Then the
size of the equivalence class containing $w$ is

$$(n-j-k)! \prod_{i \in [1, j]}\left((i-1)(c-1)+1\right)\prod_{i \in [1, k]}\left((i-1)(c-1)+1\right).$$ 
\end{thm}
\begin{proof}
By Lemmas \ref{lemsquish2} and \ref{lemclosure2}, it suffices to
count the $j,k$-squished permutations which are $j,k$-split by the
same sets as $w$. Let $J$ and $K$ be the sets such that $J$ and $K$
$j,k$-split $w$. Note that $J$ and $K$ are unique (forced by
$ck+cj+2-c \le n$).

For a permutation $w' \in S_n$ to be $j,k$-split by $J$ and $K$, the
letters in $J$ can placed in any way that makes the first $jc + 1$
letters $j$-squished. And the letters in $K$ can be placed in any way
that makes the final $kc + 1$ letters $k$-squished from right to
left. Thus by the same logic as in the proof of Theorem
\ref{thmcountsquished}, we have $$\prod_{i \in [1,
    j]}\left((i-1)(c-1)+1\right)$$ possibilities for the positions of
letters in $J$ and $$\prod_{i \in [1, k]}\left((i-1)(c-1)+1\right)$$
possibilities for the positions of the letters in $K$. Multiplying
this by the $(n - j - k)!$ ways to order the remaining letters, we get
the desired quantity.
\end{proof}

Before enumerating $g(j, k, n)$, we must define the \emph{truncated
  direct product} of the equivalence class of a word $x$ with the
equivalence class of a word $y$.
\begin{defn}
The \emph{truncated direct product} of the equivalence class of a
word $x$ with the equivalence class of a word $y$ is the set of
concatenations of elements from the first with elements from the
second, except with the following modification. Before each
concatenation, the final letter in the element from the class
containing $x$ is removed.
\end{defn}

\begin{ex} The truncated direct product of $\{123, 213\}$ with
$\{345, 354\}$ is $\{12345, 12354, 21345, 21354\}$.
\end{ex}

Now we can enumerate $g(j,k,n)$.
\begin{lem}\label{secondgcountlem}
Recall that when either both $ck+cj+2-c \le n$ and $j,k > 0$, or both
$jc+1 \ge n$ and $k=0$, or both $kc+1 \ge n$ and $j=0$, then we also
define a function $g(j, k, n)$.  We define $g(j,k,n)$ to be the number
of classes in $S_n$ comprising permutations that are $j,k$-squished
for a given choice of which $J$ and $K$ $j,k$-split them. The quantity $g(j, k, n)$ satisfies the formula
$$
 g(j, k, n) =
\begin{cases}
 (n-1)! , & \text{if } n<c+1 \\
 1, & \text{if } ck+cj+1 \ge n \ge c+1 \\
g(j+1, k, n)+g(j, k+1, n)+ \\
\sum\limits_{i=cj+2}^{n-ck-1}{g(j, 1, i) \cdot g(1,k,n-i+1) \cdot {n-j-k-1 \choose i-j-1}}, & \text{otherwise.} 
\end{cases}
$$
\end{lem}

\begin{proof}
The case of $n<c+1$ is trivial because no transformations can be
applied to permutations of size $<c+1$. The case of $cj+ck+1 \ge n \ge
c+1$ yields one class by Lemma \ref{lemsquish2} (in the case where
$j,k > 0$) and Lemma \ref{squishinglem} (in the case where $j=0$ or
$k=0$). Now, we consider the remaining case where $ck+cj+1 < n$.

Let $w$ be $j,k$-split by $A$ and $B$. Note that repeated $\{ x \mid x\in
S_{c+1},\ x_1=1 \}$ $\{ x \mid x\in S_{c+1},\ x_{c+1}=1
\}$-rearrangements on $w$ cannot move the letter $j+k+1$ between being
in the first $cj+1$ positions, in the final $ck+1$ positions, or
in-between. Indeed, this falls from Lemma \ref{lemclosure2}.

Now we consider the possibilities for $w$'s equivalence class in three
cases, when $j+k+1$ is in the first $cj+1$ letters of $w$, when
$j+k+1$ is in the final $ck+1$ letters of $w$, and when it is in
neither. Note that we have already shown above that the equivalence
class possibilities do not overlap between cases.

\textbf{Case 1:} If $j+k+1$ is in the first $cj+1$ letters of $w$,
then there are $g(j+1, k, n)$ possibilities for the equivalence class
of $w$.

\textbf{Case 2:} If $j+k+1$ is in the final $ck+1$ letters, then there
are $g(j, k+1, n)$ possibilities for the equivalence class of $w$.

\textbf{Case 3:} Suppose $j+k+1$ is neither in the first $cj+1$
letters nor the final $ck+1$ letters of $w$. Then for all $q \in S_n$
such that $w \equiv q$ under $\{ x \mid x\in S_{c+1},\ x_1=1 \}$ $\{
x \mid x\in S_{c+1},\ x_{c+1}=1 \}$, we have that $j+k+1$ acts as $1$ in
any $\{ x \mid x\in S_{c+1},\ x_1=1 \}$ $\{ x \mid x\in S_{c+1},\ x_{c+1}=1
\}$-hits containing $j+k+1$. Indeed, any letters of value less than
$j+k+1$ are kept either in the first $c(j-1)+1$ letters or final
$c(k-1)+1$ letters by Lemma \ref{lemclosure2}. Let $i$ be
the position of $j+k+1$ in $w$. It follows that the equivalence class
containing $w$ is the truncated direct product of the equivalence class
containing the first $i$ letters of $w$ with the equivalence class
containing the final $n-i+1$ letters of $w$.

There are $g(j, 1, i) \cdot g(1, k, n-i+1)$ possibilities for the pair
of equivalence classes in forming the truncated direct product. For each such
pair of equivalence classes, we also get to pick which letters are to
the left of $j+k+1$ and which are to the right in $w$. Out of the
highest $n-(j+k+1)$ letters, we get to choose $i-j-1$ of them to be to
the left of $k+j+1$ and the remaining to be to the right. Thus there
are $g(j, 1, i) \cdot g(1, k, n-i+1) \cdot {x-j-k-1
  \choose i-j-1}$ possibilities for the equivalence class containing $w$.

\textbf{Putting the cases together: } If we sum over the possibilities
for $i$ in the third case, then the total number of equivalence
classes in $S_n$ found by summing over the three cases is
$$g(j+1, k, n)+g(j, k+1, n)+\sum\limits_{i=cj+2}^{n-ck-1}{g(j, 1, i)
  \cdot g(1,k,n-i+1) \cdot {n-j-k-1 \choose i-j-1}}\text{ .}$$

The reader may wish to verify that the values of $g$ called in the
recursion stay within the domain of $g$.
\end{proof}

Finally, we can solve for the number of equivalence classes in $S_n$.
\begin{thm}\label{thmgetf2}
Recall that $f(n)$ is the number of equivalence classes in $S_n$ under $\{ x \mid x\in S_{c+1},\ x_1=1 \}$ $\{
x \mid x\in S_{c+1},\ x_{c+1}=1 \}$-equivalence. The quantity $f(n)$ satisfies
$$
 f(n) =
\begin{cases}
 n! , & \text{if }n<c+1 \\
\sum\limits_{j=1}^{n}{g(0,1,j)\cdot g(1, 0,n-j+1) \cdot {n-1 \choose j-1}}, & \text{if }n\ge c+1.
\end{cases}
$$
\end{thm}

\begin{proof}
The case of $n<c+1$ is trivial. Consider the case of $n \ge
c+1$. Under $\{ x \mid x\in S_{c+1},\ x_1=1 \}$ $\{ x \mid x\in
S_{c+1},\ x_{c+1}=1 \}$-equivalence, the position of $1$ never
changes. Furthermore, the letters to the left and right of $1$ move
independently of each other. Thus if a permutation $w$ has $1$ in
position $i$, then the equivalence class containing $w$ is the
truncated direct product of the equivalence class containing the first
$i$ letters of $w$ with the equivalence class containing the final
$n-i+1$ letters of $w$.

So, permutations with $1$ in position $i$ fall into $g(0, 1, i)\cdot
g(1, 0, n-i+1)$ classes for a given choice of which letters are to the
left and right of $1$. There are $n-1 \choose i-1$ such choices,
resulting in $g(0,1,j)\cdot g(1, 0,n-i+1) \cdot {n-1 \choose i-1}$
classes. If we sum over $i$, we obtain the desired formula.
\end{proof}

In the preceding results, we have already implicitly characterized the
equivalence classes in $S_n$ under $\{ x \mid x\in S_{c+1},\ x_1=1
\}$ $\{ x \mid x\in S_{c+1},\ x_{c+1}=1 \}$-equivalence. We formalize
this in the following remark.
\begin{rem}\label{remcount2}
Given a permutation $w$ in $S_n$ (where $n \ge c+1$), one can solve
for its class size under $\{ x \mid x\in S_{c+1},\ x_1=1 \}$ $\{ x \mid x\in
S_{c+1},\ x_{c+1}=1 \}$-equivalence (characterizing the class along
the way).

If $w$ is $n$-squished or if $w$ written from right to left (i.e.,
backwards) is $n$-squished, then Lemma \ref{squishinglem} and Lemma
\ref{lemclosurerewritten} imply that the size of the class containing
$w$ is given by Theorem \ref{thmcountsquished2}.

Otherwise, if $w$ is $j,k$-squished for $j,k > 0$ such that $cj+ck+1
\ge n$ while $ck+cj+2-c \le n$, then Lemma \ref{lemsquish2} and Lemma
\ref{lemclosure2} imply that the size of the class containing $w$ is
once again given by Theorem \ref{thmcountsquished2}.

Otherwise, let $r$ be the largest integer such that there is some $j$
and $k$ with $j+k = r$ and such that $w$ is $j,k$-squished. Because
$w$ did not fall into either of the previous two cases, we are
guaranteed that $cj+ck+1 < n$. Let $i$ be the position of $r+1$ in
$w$. In the proof of Lemma \ref{secondgcountlem}, it was shown that
the equivalence class containing $w$ is the truncated direct
product of the equivalence class containing the first $i$ letters of
$w$ with the equivalence class containing the final $n-i+1$ letters of
$w$.  Thus we can recursively compute the sizes of the two equivalence
classes in the truncated direct product to get the size of the equivalence
class containing $w$.
\end{rem}

\begin{rem}
Observe from Remark \ref{remcount2} that every equivalence class has a
size which is either 1 or is the product of values obtained in various
cases of Theorem \ref{thmcountsquished2}.
\end{rem}

\section{Conclusion and Directions of Future Work}\label{secconclusion}

It would be interesting to continue studying equivalence classes in
$S_n$ created by a replacement set comprising a pattern in $S_c$ and
its cyclic shifts (as in Section \ref{subsecrot}). When $c$ is
even, the behavior for when there is one nontrivial class versus two
nontrivial classes is still a mystery. If we focus on the case where
the pattern is $\id_c$, then when $n$ and $c$ are odd and $n >
\frac{c}{2}$, the relative sizes of the classes behave in enigmatic
patterns. In this case, the problem is equivalent to comparing the
number of even and odd hit-huggers beginning with $1$ in $S_n$ (Lemma
\ref{hit-ended --> non-avoiding}). Figure~\ref{endedtable} shows
computational results for this (collected in C++). Note that we do not
consider the case of $c=3$ because Corollary 3.2 of \cite{PRW} covers
this case in its entirety.

\begin{figure}[h]
\caption{The numbers of even and odd hit-hugging permutations in $S_n$
  beginning with $1$ where $m = \id_c $. Data is
  organized according to $n-2c$.}
\begin{center}
    \begin{tabular}{ | l | l | l | l | l |}
    \hline
     $n$ & $c$ & \# of even & \# of odd & \# of even $-$ \# of odd \\ \hline
19& 7& 2951215617& 2951215365& 252 \\ \hline \hline

13& 5& 129963& 129949& 14 \\ \hline
17& 7& 9687436& 9687390&  46 \\ \hline
21& 9& 437585672& 437585530& 142 \\ \hline \hline

11& 5& 1097& 1103& -6 \\ \hline
15& 7& 41854& 41883& -29 \\ \hline
19& 9& 1166613& 1166743& -130 \\ \hline
23& 11& 28050890& 28051452& -562  \\ \hline
     \end{tabular}
\end{center}
\label{endedtable}
\end{figure}

The following two conjectures are motivated by the data. We suspect that the data is best studied by grouping data together according to $n - 2c$.

\begin{conj}
Let $k$ be a fixed positive odd integer. Let $c$ be an arbitrary odd positive integer and $n=2c+k$. Then, the number of odd hit-huggers beginning with $1$ in $S_n$ subtracted from the number of even ones has the same sign for any pick of $c$. 
\end{conj}

\begin{conj}
Let $n, c$ be odd such that $n=2c+1$. Then, the difference between the number of even and odd hit-huggers beginning with $1$ in $S_n$ is the number of ways of choosing at most $\frac{(c-3)}{2}$ elements from a set of size $c$.
\end{conj} 

\vspace{4cm}

\textbf{Acknowledgments: }We would like to thank Professor Richard Stanley as well as the MIT PRIMES program for suggesting this area of research. We thank Sergei Bernstein and Darij Grinberg for many useful conversations throughout this research project, as well as for helping us put our thoughts on paper. Their help was invaluable. We also thank Dr. T. Khovanova for helping with the editing process of the paper.

\newpage
\bibliographystyle{plain}
\bibliography{arxiv}

\end{document}